\documentclass[11pt,draftcls,onecolumn]{IEEEtran}  
\usepackage{amssymb}
\usepackage{amsmath}
\usepackage{amsthm}
\usepackage{mathrsfs}
\usepackage{ucs}
\usepackage{subfigure}
\usepackage{multirow}
\usepackage{graphicx}
\usepackage{yhmath}
\usepackage[]{backref}
\newtheorem{sat}{{\sc Theorem}}[section]
\newtheorem{pro}[sat]{{\sc Proposition}}
\newtheorem{lem}[sat]{{\sc Lemma} }

\newtheorem{defi}[sat]{Definition}

\begin{document}
\title{Convergence of $p$-Stable  Random  Fractional Wavelet Series and Some of its Properties}
\author{Juan M. Medina,~Fernando R. Dobarro~and~Bruno~Cer\hspace{1pt}nuschi-Fr{\'\i}as
\thanks{This work was funded by the Universidad de Buenos Aires, Grant. No. 20020170100266BA, CONICET and CONAE, under Project No. 5 of the \emph{Anuncio de Oportunidad para el desarrollo de aplicaciones y puesta apunto de metodolog\'{\i}as para el \'{a}rea oceanogr\'{a}fica utlizando im\'{a}genes SAR}, Buenos Aires, Argentina.}
\thanks{J. M. Medina and B. Cernuschi-Fr\'{\i}as are with the
Universidad de Buenos Aires, Facultad de Ingenier\'{\i}a, and the Inst. Argentino  de Matem\'atica "A. P. Calder\'on", IAM, CONICET, Buenos Aires, Argentina.}
\thanks{F. R. Dobarro is with the
Universidad Nacional de Tierra del Fuego, Ant\'artida e Islas del Atl\'antico Sur, Instituto de Desarrollo Econ\'omico e Innovaci\'on, Ushuaia, Tierra del Fuego, Ant\'artida e Islas del Atl\'antico Sur, Argentina.}
}

\maketitle

\begin{abstract}
 For appropriate orthonormal wavelet basis $\{\psi_{j\,k}^e \}_{j\in\mathbb{Z}\,k\in\mathbb{Z}^d\,e\in\{0,1\}^d}$, constants $p$ and $\gamma$, if $\mathcal{I}_{\gamma}$ denotes the Riesz fractional integral operator of order $\gamma$ and $(\eta_{j\,k\,e})_{j\in\mathbb{Z} k\in\mathbb{Z}^d \,e\in\{0,1\}^d}$ a sequence of independent identically distributed symmetric $p$-stable random variables, we investigate the convergence of the series $\sum\limits_{j\,k\,e} \eta_{j\,k\,e} \mathcal{I}_{\gamma} \psi_{j\,k\,}^e$. Similar results are also studied for modified fractional integral operators. Finally, some geometric properties related to self similarity are studied.
\end{abstract}


\section{Introduction}
Uncoupled representations of random processes are of practical interest.
A classical example for Gaussian processes  is the Karhunen-L\'{o}eve (KL) representation. Motivated in part by  applications in signal and image processing \cite{Cohen,Tafti1,Unser,Van}, a usual requirement for a random process defined on $\mathbb{R}^d$ is to be  \emph{self similar} (see section \ref{secsta}) in some specified sense, since there exists several related notions in the literature.  This property, in the case $d=2$, is of certain relevance for characterizing textures. For the finite variance case, several KL like representations for the family of $\dfrac{1}{f}$ of self-similar  and related processes
were proposed, e.g. \cite{Cohen,Flan,art2,Unser} among others.
In this case, these representations have in general  the form:
\begin{equation}\label{ssf} X_{\gamma}=\sum\limits_{I} \eta_I  \mathcal{I}_\gamma \psi_I\,,
\end{equation}
where
$\mathcal{I}_\gamma$
is some fractional integration operator, $\{\psi_I\}_I$ is an orthonormal basis of $L^2(\mathbb{R}^d)$ or other Hilbert space of functions and the $\{\eta_I\}_I$ is a sequence of finite variance identically distributed random variables, in most cases Gaussian. The parameter $\gamma$ is usually linearly related to the self-similarity \emph{Hurst} parameter $H$ of the process, \cite{Fal}. Apart from applications, series like \eqref{ssf} and its geometric properties were extensively studied in the case of Fourier Gaussian random series, see for example \cite{Kah}.  Considering this sum as a \emph{generalized random process} in the sense of Gelfand and Vilenkin \cite{Gel}, Chapter 3, p. 237,
if
the $\eta_I$'s are Gaussian and $\mathcal{I}_{\gamma}$ is the Riesz fractional integration operator (Definition \ref{Riesz}) then this sum converges a.s. in the sense of distributions, i.e. in $ \mathcal{D}^{\prime}(\mathbb{R}^d)$  to a self-similar process as defined here in Section \ref{secsta} in terms of equality in probability law between $X_{\gamma}$ and a re-scaled version of it: $a^{\delta}X_{\gamma}(a\,.\,)$ for some $\delta\in\mathbb{R}$. In this particular case, $X_{\gamma}$ is a fractional Gaussian noise (See Theorem \ref{T2}). These type of representations have received some interest because of  its simplicity for modeling certain random signals (see e.g. \cite{Unser}), since one only needs to know the probability distribution of the coefficients $\eta_I$ and the parameter $\gamma$ or similar. On the other hand, the finite variance requirement may be a constraint in some applications. A first attempt to overcome this limitation, retaining at the same time some of the properties of interest of $X_{\gamma}$, is to substitute the $\eta_I$'s with non Gaussian $p$-stable random variables, $p\in(0,2)$, \cite{Sam}.
However, it may become a non trivial task to check which properties are preserved for this case. For example, besides self similarity, in \cite{Pip} is proved that it is not possible to represent a $p$-stable stationary random process by a series like
\eqref{ssf}.

Here, we prove that for appropriate parameters $\gamma \leq\dfrac{d}{2}$ and $p$, if we consider $\{\psi_I\}_I$ a suitable wavelet basis, the series \eqref{ssf} stills converges a.s.  in $ \mathcal{D}^{\prime}(\mathbb{R}^d)$, and if we change $\mathcal{I}_{\gamma}$ by a modified operator, then it converges to an ordinary process for the case $\dfrac{d}{2}<\gamma \leq \dfrac{d}{2}+1$. If $p=2$ the limit of the series \eqref{ssf} is self similar of parameter $\dfrac{d}{2}+\gamma$,  and in the case $p\neq 2$, although its limit is not necessarily self similar, we can prove that the distribution function of  the
re-scaled
process $a^{\frac{d}{2}+\gamma}X_{\gamma}(a\,.\,)$ is, in some sense, properly stochastically dominated. In the Gaussian case of $p=2$, the series of equation \eqref{ssf} converges to a fractional Gaussian noise, for which  an integrated version of it gives the well known fractional Brownian motion, and its $d$-dimensional analogues, with their known ``fractal'' properties.  We shall see that, for appropriate parameters $p$ and $\gamma$, that integrated versions of the process $X_{\gamma}$  have a graph with  Hausdorff dimension greater than $d$, justifying   the possible use of the process defined by \eqref{ssf} as a model of a \emph{fractal} process still for $p\neq 2$.

\section{Auxiliary results and definitions.}\label{secaux}

\subsection{Function spaces, Fourier transforms and Wavelets.}\label{auxfunc}

In the following, if $p\in[1,\infty]$ and  $\mu$  is the Borel measure over $\mathbb{R}^d$,  the corresponding Lebesgue  spaces of the equivalence classes of functions will be denoted by $L^p (\mathbb{R}^d, d\mu)$,
and if $\mu$ is the usual Lebesgue measure, we will write shortly $L^p (\mathbb{R}^d)$. When $p=2$ it becomes a Hilbert space  and the  $L^2(\mathbb{R}^d)$ inner product will be denoted by $\langle{\,.\,,.\,}\rangle$. If $x\in \mathbb{C}^d \;  (d\in\mathbb{N})$ we will denote its  usual norm by $\left|x\right|$ and the support of a function $f$ is defined by $supp(f)=\overline{\left\{x\,:f(x)\neq 0\right\}}$.
The Schwartz class of functions $\mathcal{S}(\mathbb{R}^d)$ is defined as the linear space of smooth functions rapidly decreasing at infinity, together with its derivatives. This means that $\phi \in \mathcal{S}(\mathbb{R}^d)$ whenever
$
 \phi  \in C^\infty  \left( {\mathbb{R}^d } \right)$ and
 $$
 \mathop {\sup }\limits_{(x_1 ,...x_d ) \in \mathbb{R}^d }
\prod\limits_{i = 1}^d {\left| {x_i } \right|^{\alpha _i } }
\left| {\frac{\partial } {{\partial x_1^{\beta _1 }
}}...\frac{\partial } {{\partial x_d^{\beta _d } }}\phi (x_1
,...x_d )} \right|< \infty\; \forall \; \alpha _j \, \beta _j
\,  \in \, \mathbb{N}\,,
$$
endowed with its usual topology.
We will denote $\mathcal{D}(\mathbb{R}^d)$ the space of functions which are in $C^\infty  \left( {\mathbb{R}^d } \right)$ and have compact support. Both spaces are topological vector spaces, for more details see \cite{Gra}, Chapter 2, p. 109. Their duals are denoted as: $\mathcal{S}^{\prime}(\mathbb{R}^d)$ (\textit{Tempered distributions}) and $\mathcal{D}^{\prime}(\mathbb{R}^d)$ (\textit{distributions}) respectively. Clearly:
$\mathcal{D}(\mathbb{R}^d)\subset \mathcal{S}(\mathbb{R}^d)$ and then $\mathcal{S}^{\prime}(\mathbb{R}^d)\subset \mathcal{D}^{\prime}(\mathbb{R}^d)$.
The \textit{Fourier Transform} $\widehat{f}$ of $f\in \mathcal{S}(\mathbb{R}^d)$ is defined as
$
\widehat{f} \left( \lambda  \right) = \int\limits_{\mathbb{R}^d } {f \left( x \right)} e^{ - 2\pi i\lambda .x} dx\,.
$
It is a known fact that $\widehat{f}$ also belongs to the space  $\mathcal{S}(\mathbb{R}^d)$.  The Fourier transform can be defined, as usual as a linear map over $L^1 (\mathbb{R}^d)$, 
 as an isometry on $L^2(\mathbb{R}^d)$ or  over the class of tempered distributions. The inverse Fourier transform $\mathop{f}\limits^{\vee}$ is defined in an analogous way. For further references on Fourier transforms and series, see for example  \cite{Gra}.

Below, we will need a variant of the classical Shannon, Nyquist and Kotelnikov sampling theorem.
%
 \hfill\\
\begin{sat}\label{t-1} If
$ f \in L^2 (\mathbb{R}^d ) $  is such that
$supp (f) \subset [ - x_o , x_o ]^d$
 with
$
x_o  < \dfrac{1}{2}
$. Then there exists $ \phi  \in \mathcal{S}(\mathbb{R}^d ) $
such that

\begin{equation}\label{shan}
\widehat{f}(\lambda ) = \sum\limits_{k \in\mathbb{Z}^d }
{\widehat{f}(k)} \phi (\lambda  - k)
\end{equation}
\end{sat}

\begin{proof} Let
$ \widetilde{f}(x) = \sum\limits_{k \in \mathbb{Z}^d } {f(x + k)}
$
 be the periodization of $f$. Then, $\widetilde{f}$  verifies
$$
\widetilde{f} \in L^2 \left(\left[-\dfrac{1}{2},\dfrac{1}{2}\right]^d \right) \subset L^1 \left(\left[-\dfrac{1}{2},\dfrac{1}{2}\right]^d \right)
$$ and therefore
$\widetilde{f} $ has Fourier series given by  $$\sum\limits_{k \in \mathbb{Z}^d } {a_k e^{ -
2\pi ix.k} },$$ and then
$
\mathop {\lim }\limits_{R  \to \infty } \sum\limits_{k \in D_R  } {a_k e^{ - 2\pi ix.k} }  =  \widetilde{f} $ a.e. and in $
L^1 \left(\left[-\dfrac{1}{2},\dfrac{1}{2}\right]^d \right)
$  (and in $L^2$) norm for a suitable  domain $
{D_R  } \in \mathbb{R}^d $.
%
%
Next, we can take $\phi  \in \mathcal{S}(\mathbb{R}^d )$ such that
$$
\mathop{\phi}\limits^{\vee}
(x ) = \left\{ \begin{array}{l}
  1,\left| {x _i } \right| < x_0 \\
  0,\left| {x _i } \right| \geq 1-x_0 \\
 \end{array} \right.\,.
$$
 Defining  $
S_R  (x) =\mathop{\phi}\limits^{\vee}(x)\left({\sum\limits_{k \in D_R
} a_k e^{ - 2\pi ix.k} }\right)
$,
 then
$
f = \widetilde{f}\;\mathop{\phi}\limits^{\vee}
$ and $\mathop {\lim }\limits_{R  \to \infty } \left\| {S_R   - f} \right\|_{L^1 (\mathbb{R}^d )}  = 0$.
This implies
$$
\mathop {\lim }\limits_{R  \to \infty } \mathop {\rm{sup}}\limits_{\lambda  \in \mathbb{R}^d } \left| {\widehat{S_R  }(\lambda ) - \widehat{f}(\lambda )} \right| =
0\,,
$$
%
%
but (see e.g. \cite{Gra}, Exercise 3.6.4, p.236)
$
a_k  = \widehat{f}(k)
$, so that
$$
\widehat{S_R  }(\lambda ) = \sum\limits_{k \in D_R  }
{\widehat{f}(k)} \phi (\lambda  - k)\;.
$$
Then \eqref{shan} follows immediately from this.
\end{proof}
In the following we will use fractional integral operators, for which some of their properties are reviewed.
We begin with a definition (\cite{Gra2}, Chapter 6, p. 2 or \cite{Stein}, Chapter 5, p. 117):
\begin{defi} Let $0<\alpha<d$.
For $ f \in \mathcal{S}(\mathbb{R}^d )$ we define its
Riesz Potential:

\begin{equation}\label{Riesz}
(\mathcal{I}_{\gamma} f)(x) = \frac{1} {{C_\gamma }} \int\limits_{\mathbb{R}^d }
\frac {f(y)} {\left| {x - y} \right|^{d -\gamma }} \,  dy
\end{equation}
where
$ C_\gamma  = \dfrac{{\pi ^{d/2} \, 2^\alpha \,
\Gamma \left( {\dfrac{\gamma } {2}} \right)}} {{\Gamma \left(
{\dfrac{d} {2} - \dfrac{\gamma } {2}} \right)}} $.
\end{defi}
Riesz potentials have the following scaling property: for every $a\neq 0$: $\mathcal{I}_{\gamma}(f(a\,.\,))=|a|^{-\gamma} (\mathcal{I}_{\gamma}f)(a\,.\,)$,
i.e.
$\mathcal{I}_{\gamma}(f(a\,y\,))(x)=|a|^{-\gamma} (\mathcal{I}_{\gamma}f(y))(a\,x\,)$.
A crucial result for this integral operator is the following, \cite{Gra2}, Chapter 6, p.3 :

\begin{sat}\label{t0}{(Hardy, Littlewood and Sobolev)} Let $
0 < \gamma  < d$, $1\leq p < q < \infty$ and $
\dfrac{1} {q} = \dfrac{1} {p} - \dfrac{\gamma } {d} $ then: \hfill\\
(a) For all $ f \in  L^p (\mathbb{R}^d ) $, the integral that
defines
$  \mathcal{I}_{\gamma} f$ converges a.e. \hfill\\
(b)If $ p > 1$ then
\begin{equation}
\left\| { \mathcal{I}_{\gamma} f} \right\|_{L^q (\mathbb{R}^d)}
\leq C_{pq} \left\| f \right\|_{L^p (\mathbb{R}^d)} \,.
\end{equation}
\end{sat}
Note that, in the appropriate sense, the Fourier Transform of $ \mathcal{I}_{\gamma}f$ is given by:
\begin{equation}\label{fourfrac}\widehat{\mathcal{I}_{\gamma}f} (\lambda)= (2\pi )^{ - \gamma
} \left| \lambda \right|^{ - \gamma } \widehat{f}(\lambda)
\end{equation}
and it is easy to check that for $f
\in\mathcal{S}(\mathbb{R}^d )$ and $\alpha+\beta<d$ then $ \mathcal{I}_{\alpha}(\mathcal{I}_{\beta} f) = \mathcal{I}_{\alpha  +
    \beta} (f)$. Furthermore, if $\rm \Delta \it f = \sum\limits_{j = 1}^d {\dfrac{{\partial ^2 f}}
{{\partial x_j^2 }}}$ is the Laplacian of $f$ , then
  $\Delta ( \mathcal{I}_{\gamma} f) =  \mathcal{I}_{\gamma-2} f$.
Finally, $\mathcal{I}_{\gamma}$ can be thought as defined by the convolution with the locally integrable function $k_{\gamma}(x)=\dfrac {1} {C_\gamma} \dfrac{1}{\left| {x} \right|^{  d - \gamma }}$, and is formally self adjoint, in the sense that for every $f,g\,\in\,\mathcal{S}(\mathbb{R}^d)$:
\begin{equation}\label{selfadj}
\langle{\mathcal{I}_{\gamma}f,g}\rangle=\langle{f,\mathcal{I}_{\gamma}g}\rangle\,.
\end{equation}
  Considering again $k_{\gamma}$, we can define a fractional integral operator for $f\in L^p (\mathbb{R}^d)$,  in the following way:
$$\mathcal{K}_{\gamma}f(x)=\int\limits_{\mathbb{R}^d} (k_{\gamma}(x-y)-k_{\gamma} (y))f(y) dy=\int\limits_{\mathbb{R}^d} K_{\gamma}(x,y)\, f(y) dy$$
The modified kernel $K_{\gamma} (x,y)=k_{\gamma}(x-y)-k_{\gamma} (y)$ is easier to control, and we sketch the proof of
the following lemma:
\begin{lem}\label{ker} If $1<p<\infty $ and $0< d\left({1-\dfrac{1}{p}}\right)<\gamma<d\left({1-\dfrac{1}{p}}\right)+1$, then
$K_{\gamma}(x,\,.\,)\in L^p(\mathbb{R}^d)$ and moreover:\hfill\\
(i) There exists a positive constant $C_{p\,\gamma\,d}$ such that for each $x\in\mathbb{R}^d$:
$$\left\|{K_{\gamma}(x,\,.\,)}\right\|_{L^p(\mathbb{R}^d)} = C_{p\,\gamma\,d}\, |x|^{\gamma-\left({1-\frac{1}{p}}\right)d}\,.$$
(ii) For every $x,x'\in\mathbb{R}^d$: $\left\|{K_{\gamma}(x,\,.\,)-K_{\gamma}(x',\,.\,)}\right\|_{L^p(\mathbb{R}^d)}=\left\|{K_{\gamma}(x-x',\,.\,)}\right\|_{L^p(\mathbb{R}^d)}$.
\end{lem}
\begin{proof} (Sketch) Since
$$\left\|{K_{\gamma}(x,\,.\,)}\right\|_{L^p(\mathbb{R}^d)} ^p =\int\limits_{\{|y|<2|x|\}} |K_{\gamma} (x,y)|^p dy +\int\limits_{\{|y|\geq 2|x|\}} |K_{\gamma} (x,y)|^p dy\,.$$
The  condition $d\left({1-\dfrac{1}{p}}\right)<\gamma $ gives the appropriate exponent for the boundedness of the first integral.
In addition, since $\gamma<d\left({1-\dfrac{1}{p}}\right)+1$ and considering that for some positive constant $C$
$$|K_{\gamma} (x,y)| \leq C |x-y|^{\gamma-d-1} |x|\,,$$
if $|y|>2|x|$, then the second integral is also finite. Hence, the map $x \mapsto \left\|{K_{\gamma}(x,\,.\,)}\right\|_{L^p(\mathbb{R}^d)}$ is well defined and by a change of variable, we obtain that it is an homogeneous function depending only on $|x|$, from which assertion (i) follows. Assertion (ii) is also obtained by a change of variable.
\end{proof}
For fixed $x\in\mathbb{R}^d$, we note that in the Fourier domain $\mathcal{K}_{\gamma}$ can be characterized, in an appropriate sense, \cite{Cohen},  Chapter 3, p. 45, by:
\begin{equation}\label{intfrac}
\mathcal{K}_{\gamma}f(x)=\frac{1}{(2\pi)^\gamma} \int\limits_{\mathbb{R}^d} \left({\frac{e^{-2\pi i\lambda x} -1}{ \left| \lambda \right|^{  \gamma} }}\right)\widehat{f}(\lambda) d\lambda\,.
\end{equation}
Some formal manipulations show that from equations \eqref{fourfrac} and \eqref{intfrac}, for suitable parameters $\beta$ and $\gamma$, we have:
\begin{equation}\label{combfrac}
\widehat{(\mathcal{I}_{\gamma} K_{\beta}(x,\,.\,))}(\lambda)=\widehat{K_{\beta+\gamma} (x,\,.\,)}(\lambda)=\frac{1}{(2\pi)^{\gamma+\beta}} \left({ \frac{e^{-2\pi i\lambda x} -1}{ \left| \lambda \right|^{  \beta} }}\right)\frac{1}{|\lambda|^{\gamma}}.
\end{equation}
and
\begin{equation}\label{combfrac2}
\mathcal{K}_{\gamma}(\mathcal{I}_{\beta}f)(x)=\mathcal{K}_{\beta+\gamma} f (x)=\int\limits_{\mathbb{R}^d} K_{\beta+\gamma} (x,y)f(y) dy\,.
\end{equation}
For $s\in\mathbb{R}$ another related operator $\mathcal{J}_s f$ is defined, formally, by its Fourier transform as:
\begin{equation}\label{Bessel}
\widehat{\mathcal{J}_s f} (\lambda)= (1+|\lambda|^2)^{s/2}\widehat{f}(\lambda) \,.
\end{equation}
\begin{sat}\label{t3bis} \cite{Gra2}, Chapter 6, p. 8. If $s<0$ and $p\geq 1$, $\mathcal{J}_s :L^p (\mathbb{R}^d)\longrightarrow L^p (\mathbb{R}^d)$ defines a continuous linear operator, i.e. there exists $C_p>0$ such that
$$\left\| {\mathcal{J}_s f} \right\|_{L^p (\mathbb{R}^d) }
\leq C_{p} \left\| f \right\|_{L^p (\mathbb{R}^d) } \,.$$
\end{sat}
\noindent
For $1<p<\infty$, and $s\in\mathbb{R}$, we introduce the \emph{Sobolev spaces} $H^p _s (\mathbb{R}^d)$:
 $$H^p _s (\mathbb{R}^d)=\left\{{f\in\mathcal{S}'(\mathbb{R}^d):\;\;\mathcal{J}_s f\in \,L^p (\mathbb{R}^d)}\right\}\,.$$
 These are Banach spaces of tempered distributions with the norm defined by  $\left\|{f}\right\|_{H^p _s (\mathbb{R}^d)}=\left\|{J_s f}\right\|_{L^p (\mathbb{R}^d)}$. Moreover, \cite{Mey1992}, p.168, if $s\geq 0$, this norm is equivalent to $\left\|{f}\right\|_{L^p(\mathbb{R}^d)} +\left\|{(|\,.\,|^s \widehat{f})^\vee}\right\|_{L^p  (\mathbb{R}^d)}$. Recalling again equation \eqref{intfrac} the equivalence of norms for $K_{\gamma}(x,\,.\,)$ takes the following form which will be useful in the sequel:
 \begin{equation}\label{normadeK}
  \left\|{K_{\gamma}(x,\,.\,)}\right\|_{H^p _s (\mathbb{R}^d)}\sim \left\|{K_{\gamma}(x,\,.\,)}\right\|_{L^p(\mathbb{R}^d)} +\left\|{(K_{\gamma-s}(x,\,.\,)}\right\|_{L^p  (\mathbb{R}^d)}\,.
 \end{equation}
 In the particular case $s=-d$, only when $p=2$, the $H^p _s (\mathbb{R}^d)$ spaces coincide with the following $\mathcal{F}{L^p}_w$ spaces, which are introduced for auxiliary purposes.
\begin{pro}\label{Fp} For $1 \leq p \leq 2$, the space
$$\mathcal{F}{L^p}_w=\left\{{f\in\mathcal{S}'(\mathbb{R}^d):\;\;\widehat{f}(1+|\,.\,|^2)^{-d}\in\,L^p (\mathbb{R}^d)}\right\}$$
is a Banach space with the norm defined by $\left\|{f}\right\|_{\mathcal{F}{L^p}_w}=\left\|{\widehat{f}(1+|\,.\,|^2)^{-d}}\right\|_{L^p (\mathbb{R}^d)}$. Moreover convergence in $\mathcal{F}{L^p}_w$ implies convergence in $\mathcal{S}'(\mathbb{R}^d)$.
\end{pro}
\begin{proof}
Observe that if we define $w(\lambda)=(1+|\lambda|^2)^{-d}$, then $f\in\mathcal{F}{L^p}_w$ if and only if $\widehat{f}\in\;L^p(\mathbb{R}^d, w \, d\lambda)$.  Let $(f_n)_{n\in\mathbb{N}}$ be a Cauchy sequence en $\mathcal{F}{L^p}_w$ which is equivalent to $(\widehat{f}_n)_{n\in\mathbb{N}}$
being a Cauchy sequence in $L^p(\mathbb{R}^d, w \, d\lambda)$, and then there exists a unique $g\in L^p(\mathbb{R}^d, wd\lambda)$ such that
$\left\|{\widehat{f}_n-g}\right\|_{L^p (\mathbb{R}^d, w d\lambda)}\longrightarrow 0$, when $n\longrightarrow\infty$. We shall verify that $g\in\mathcal{S}'(\mathbb{R}^d)$ and therefore taking $f:=g^{\vee} \in \mathcal{S}'(\mathbb{R}^d)$ we are done.
For this 
take $\dfrac{1}{p} +\dfrac{1}{q}=1$ and $m>d\left(1+2\dfrac{q}{p}\right)$ then by H\"{o}lder's inequality:

$$\int\limits_{\mathbb{R}^d} \frac{|g(\lambda)|}{(1+|\lambda|)^m} d\lambda
=\int\limits_{\mathbb{R}^d} \frac{|g(\lambda)|}{(1+|\lambda|)^m} \frac{(1+|\lambda|^2)^{\frac{d}{p}}}{(1+|\lambda|^2)^{\frac{d}{p}}} d\lambda  $$
$$\leq  \left({ \int\limits_{\mathbb{R}^d} {|g(\lambda)|^p}{(1+|\lambda|^2)^{-d}}  d\lambda}\right)^{\frac{1}{p}}\left({ \int\limits_{\mathbb{R}^d} \frac{(1+|\lambda|^2)^{\frac{dq}{p}}}{(1+|\lambda|)^{m q}}  d\lambda}\right)^{\frac{1}{q}} <\infty\,,$$
thus, see e.g. \cite{Gra}, Exercise 2.3.1, p.122, $g\in \mathcal{S}'(\mathbb{R}^d)$ and therefore $f\in\mathcal{F}{L^p}_w$. Finally,  $\mathop{f_n \longrightarrow f}\limits_{n\longrightarrow\infty}$ in $\mathcal{F}{L^p}_w$ if and only if $\mathop{\widehat{f}_n \longrightarrow \widehat{f}}\limits_{n\longrightarrow\infty}$ in $L^p (\mathbb{R}^d, w \, d\lambda)$. Let $\varphi\in \mathcal{S}(\mathbb{R}^d)$, then, if $\dfrac{1}{p}+\dfrac{1}{q}=1$, by definition of the Fourier Transform of a tempered distribution and H\"{o}lder's inequality we get:
$$|\langle{f_n,\varphi}\rangle-\langle{f,\varphi}\rangle|=|\langle{\widehat{f}_n-\widehat{f},\varphi^{\vee}}\rangle|=\left|{\int\limits_{\mathbb{R}^d}(\widehat{f}_n (\lambda)-\widehat{f} (\lambda))\varphi^{\vee}(\lambda) d\lambda}	\right|$$
$$=\left|{\int\limits_{\mathbb{R}^d}(\widehat{f}_n (\lambda)-\widehat{f} (\lambda))\varphi^{\vee}(\lambda)\frac{(1+|\lambda|^2)^{\frac{d}{p}}}{(1+|\lambda|^2)^{\frac{d}{p}}} d\lambda}	\right|$$
$$\leq \left({\int\limits_{\mathbb{R}^d}|\widehat{f}_n (\lambda)-\widehat{f} (\lambda)|^p \frac{1}{(1+|\lambda|^2)^d} d\lambda} 	\right)^{\frac{1}{p}}  \left({\int\limits_{\mathbb{R}^d}|\varphi^{\vee}(\lambda)|^q (1+|\lambda|^2)^{\frac{dq}{p}} d\lambda}	\right)^{\frac{1}{q}}\,,$$
which proves the last assertion of Proposition \ref{Fp}.
\end{proof}
The following estimate for the $ \mathcal{F}{L^p}_w$ norm will be useful in the sequel.
\begin{lem}\label{L1}
Let $1\leq p \leq 2$, then  $L^2(\mathbb{R}^d)\subset \mathcal{F}{L^p}_w$ and moreover, if $Q=\left[{-\dfrac{1}{4},\dfrac{1}{4}}\right)^d$, there exits a positive constant $C_{p\,d}$ such that for every  $f\in L^2(\mathbb{R}^d)$, $f=0$ a.e. in $Q^c$, the following inequality holds:
\begin{equation}\label{ineq1}
\left\|{f}\right\|_{\mathcal{F}{L^p}_w} ^p \leq C_{p\,d} \sum\limits_{k\in\mathbb{Z}^d} |\widehat{f}(k)|^p (1+|k|^2)^{-d}\,.
\end{equation}
\end{lem}
\begin{proof} If $p=2$ the result is immediate. To prove the first assertion for $p\neq 2$, by H\"{o}lder's inequality one has the following estimate
$$\left\|{f}\right\|_{\mathcal{F}{L^p}_w} ^p\leq \left\|{f}\right\|_{L^2 (\mathbb{R}^2)} ^p \left({\int\limits_{\mathbb{R}^d} \frac{d\lambda}{(1+|\lambda|^2)^{d/(1-\frac{p}{2})}}}\right)^{1-\frac{p}{2}}\,.$$
For the second assertion, under these conditions we can write $$\widehat{f}(\lambda)=\sum\limits_{k\in\mathbb{Z}^d} \widehat{f}(k) \phi(\lambda-k),$$ as in Theorem \ref{t-1} and therefore:
\begin{equation*}
\left\|{f}\right\|_{\mathcal{F}{L^p}_w}=\int\limits_{\mathbb{R}^d} |\widehat{f}(\lambda)|^p (1+|\lambda|^2)^{-d} d\lambda
\end{equation*}
\begin{equation*}
\leq \int\limits_{\mathbb{R}^d} \left({\sum\limits_{k\in\mathbb{Z}^d} |\widehat{f}(k)| |\phi(\lambda-k)| (1+|\lambda|^2)^{-d/p}}\right)^p d \lambda
\end{equation*}
\begin{equation}\label{e1}
\leq \int\limits_{\mathbb{R}^d} \left({\sum\limits_{k\in\mathbb{Z}^d} |\widehat{f}(k)| |\phi(\lambda-k)| 2^{d/p} (1+|k|^2)^{-d/p}(1+|\lambda-k|^2)^{d/p}}\right)^p d \lambda\,
\end{equation}
since $(1+|\lambda|^2)^{-d} \leq 2^{d} (1+|k|^2)^{-d}(1+|\lambda-k|^2)^{d} $ by Peetre's inequality.
If $\dfrac{1}{p}+\dfrac{1}{q}=1$, take $a_k(\lambda)=|\phi(\lambda-k)|^{\frac{1}{q}}$ and
$$b_k(\lambda)=|\widehat{f}(k)| |\phi(\lambda-k)| 2^{\frac{d}{p}} (1+|k|^2)^{-\frac{d}{p}}(1+|\lambda-k|^2)^{\frac{d}{p}}|\phi(\lambda-k)|^{\frac{1}{p}},$$
by H\"{o}lder's inequality we get:
\begin{equation}\label{e2}
\left\|{f}\right\|_{\mathcal{F}{L^p}_w}\leq \int\limits_{\mathbb{R}^d} \sum\limits_{k\in\mathbb{Z}^d} |b_k(\lambda)|^p \left({\sum\limits_{k\in\mathbb{Z}^d} |a_k(\lambda)|^q}\right)^{\frac{p}{q}}  d \lambda\,,
\end{equation}
finally, since there exists  some positive constant $C$ such that:
\begin{equation*}
\sum\limits_{k\in\mathbb{Z}^d} |a_k(\lambda)|^q = \sum\limits_{k\in\mathbb{Z}^d} |\phi(\lambda-k)|\leq C,
\end{equation*}
then equation \eqref{e2} becomes
$$\leq 2^d C \int\limits_{\mathbb{R}^d} \sum\limits_{k\in\mathbb{Z}^d} |\widehat{f}(k)|^p (1+|k|^2)^{-d}(1+|\lambda-k|^2)^{d} |\phi(\lambda-k)| d\lambda$$
$$=2^d C \int\limits_{\mathbb{R}^d} (1+|\lambda|^2)^{d} |\phi(\lambda)| d\lambda \sum\limits_{k\in\mathbb{Z}^d} |\widehat{f}(k)|^p (1+|k|^2)^{-d}\,.$$
\end{proof}

\subsection{Some probability, stable laws and generalized random processes.}\label{secsta}
Let $(\Omega,\mathcal{F},\mathbf{P})$ be a probability space and $X$ a random variable variable defined on it. The \emph{distribution function of} $X$ is defined, for $x\in\mathbb{R}$, as $F_X (x)=\mathbf{P}(X\leq x)$. If $\varphi$ is any Borel measurable real function, we will denote the expectation of $\varphi(X)$ with $\mathbf{E}(\varphi(X))$. The characteristic function of $X$ is
$\Phi_X (\xi)=\mathbf{E}(e^{i\xi X})$. For $p\in(0,2]$, we say that a random variable $\eta$ is symmetric $p$-stable of parameter $\sigma>0$ if $\Phi_{\eta}(\xi)=e^{-\sigma^{p}|\xi|^p}$. A symmetric $p$-stable random variable $\eta$ will be denoted as $\eta\sim SpS$. When we write $F_{\eta_p}$ we shall be referring to the distribution function of such a random variable with $\sigma=1$. Note that $p=2$ corresponds to the Gaussian case and therefore $\eta\sim\mathcal{N}(0,\sigma)$. 
Let us review some basic properties of stable distributions, see \cite{Sam}, Chapter 1, p. 10, and \cite{Kwo92}, Chapter 0, p.5.
\begin{enumerate}\label{stable}
 \item
 If $\eta_1,\dots, \eta_n$ are independent and $\eta_i \sim SpS$, with parameter $\sigma_i$ then $\sum\limits_{i=1}^n \eta_i \sim SpS$, with $\sigma' = \left\|(\sigma_{\eta_i})_i \right\|_{l^{p}}$.\hfill\\
\item Let $p < 2$. If $\eta\sim SpS$  and $0<r<p$ then $(\mathbf{E}|\eta|^r)^{1/r} = \, C_r \, \sigma_\eta$, where ${C_{r}}^r = \mathbf{E} |\eta_p|^{r}$, and $\mathbf{E}|\eta|^r =\infty$ for $r\geq p$. 
\end{enumerate}
Let $\mu$ be a non negative Borel measure on $\mathbb{R}^d$. We shall need a result on the a.s. convergence of random elements in
$L^r (\mathbb{R}^d,d\mu)$. This theorem is a particular case of a more general one in \cite{Kwo92}, Chapter 2.
\begin{sat}\label{Stoconv}
Let $0<r<p<2$, $\{f_j\}_{j\in\mathbb{N}} \subset L^r (\mathbb{R}^d,d\mu)$, and let $\{\eta_j\}_{j\in\mathbb{N}}\sim SpS$ be a sequence  of independent and identically distributed random variables. Then the series $\sum\limits_{i=1}^{\infty} \eta_i f_i$ converges in $L^r (\mathbb{R}^d,d\mu)$ a.s. if and only if
$$\left\|{\left({\sum\limits_{i=1}^{\infty}|f_i|^{p} }\right)^{1/p}}\right\|_{ L^r (\mathbb{R}^d,d\mu)} <\infty\,.
$$
\end{sat}

Our results, are aimed at the construction of certain random variables taking values in $\mathcal{D}' (\mathbb{R}^d)$. In this case, every $\mathcal{D}' (\mathbb{R}^d)$- valued random variable, say $X$, takes the form of a random linear functional defined on $\mathcal{D} (\mathbb{R} ^d)$. Previously, we will also need to define the class of \textit{generalized random  processes}, of which these $\mathcal{D}' (\mathbb{R}^d)$- valued random variables are particular cases. Following \cite{Gel}, Chapter 3, p. 237,  and \cite{Unser}, Chapter 4, p. 57,  we will say that a generalized random functional is  defined on $\mathcal{D}(\mathbb{R}^d)$ if for every $\varphi\,\in\,\mathcal{D} (\mathbb{R}^d)$ there is associated a real valued random variable $X(\varphi)=\langle{X,\varphi}\rangle$.
 In accordance with the usual specification of the probability distributions of a countable set of real random variables, given $n\in\mathbb{N}$, $\varphi_1, \dots, \varphi_n\,\in\, \mathcal{D}(\mathbb{R}^d)$ define the probability of the events,
$\{ a_k \leq \langle{X,\varphi _k}\rangle < b_k \},\,\,\;k=1,\dots, n\;,$ which will have to be compatible in the usual sense.
 On the other hand, linearity means that for any $a,b \in \mathbb{R}$, $\varphi,\,\psi\in \mathcal{D} (\mathbb{R}^d)$:
 $\langle{X,a \varphi + b \psi}\rangle= a \langle{X,\varphi}\rangle + b \langle{X,\psi}\rangle \;\;\; \textrm{a.s.}$.
 For a comprehensive study on this topic, see \cite{Gel}.
 In an analogous way to real valued random variables, for each $\varphi\in \mathcal{D} (\mathbb{R}^d)$ we can calculate the characteristic function of the real random variable $\langle{X,\varphi}\rangle $, $\Phi_{\langle{X,\varphi}\rangle} (\xi)=\mathbf{E}(e^{i\xi \langle{X,\varphi}\rangle})$. In fact if $\xi=1$ and considering $\varphi$ as a variable, this gives the \emph{characteristic functional} of $X$, $\Phi_X (\varphi)=\mathbf{E}(e^{i\langle{X,\varphi}\rangle})$,
which completely determines its distributions as in the  case of ordinary random processes. Finally, \emph{self-similarity} for generalized random processes can be defined in the following analogous way to \cite{Unser}, p. 178: $X$ is self-similar if there exists  a constant $\delta>0$  such that
\begin{equation} \label{self}
\Phi_X (\varphi)=\Phi_X (a^{\delta}\varphi(a\,.\,))\,,
\end{equation}
 for every dilation factor $a>0$ and $\varphi \in  \mathcal{D} (\mathbb{R}^d)$. This means that $X$ is equivalent, in probability law, to $a^r X(\,.\,/a)$, for some appropriate constant $r$. In this context, we recall the Hausdorff dimension, see \cite{Fal}, Chapter 2, p. 21, of a subset $A$ of $\mathbb{R}^d$ denoted by $dim_H(A)$. Although self similarity is associated to the notion of ``fractality'', the last one has not a precise meaning. However, subsets of $\mathbb{R}^d$  with non integer Hausdorff dimension    are considered as displaying a fractal behaviour. A  way for the study of the fractal behaviour of the graph of a function is the calculation of  its Hausdorff dimension. Usually, the estimation of a lower bound for this value is calculated by potential methods, see \cite{Fal}, Chapter 2, p. 26, and \cite{Kah}, Chapter 10, p.132.  An example is:
\begin{lem}\label{Frost} If $B$ is a compact subset of $\mathbb{R}^d$ and $\mathcal{G}\subset\mathbb{R}^{d+1}$ denotes the graph of a measurable  function $f:B\longrightarrow\mathbb{R}$ and
$\int\limits_B \int\limits_B (|x-x'|^2 +|f(x)-f(x')|^2)^{-\rho/2} \, dx \, dx' <\infty$ then $dim_H (\mathcal{G})>\rho$.
\end{lem}
Other related results will be introduced in the final section, for the estimation of the Hausdorff dimension of certain processes arising from the construction introduced in equation \eqref{ssf}.
\subsection{Wavelets.}
Let $\{\psi_{j\,k} ^e\}_{j\in\mathbb{Z}\,k\in\mathbb{Z}^d\,e\in E}$, with $E=\{0,1\}^d$, be an orthonormal wavelet basis of $L^2(\mathbb{R^d})$, \cite{Mey1992}, Chapter 2. The Parseval identity for this case is:
\begin{equation}\label{pars}
\left\|f\right\|^2 _{L^2 (\mathbb{R}^2)}=\sum\limits_{e \in E}\sum\limits_{j\in\mathbb{Z}}\sum\limits_{k\in\mathbb{Z}^d} |\langle{f,\psi_{j\,k} ^e}\rangle|^2\,.
\end{equation}
Therefore the norm  $\left\|f\right\|^2 _{L^2 (\mathbb{R}^2)}$ can be estimated from the wavelet coefficients $\langle{f,\psi_{j\,k} ^e}\rangle$. Under some additional conditions, for example if the wavelet basis arises from a \emph{$r$-regular wavelet multirresolution approximation} of $L^2(\mathbb{R}^d)$, then, if $\{I_{j\,k}\}_{j\in \mathbb{Z},k\in\mathbb{Z}^d}$ denotes the family of dyadic cubes of $\mathbb{R}^d$, for some positive constants $c_p,c_{p\,s}, C_p, C_{p\,s}$,  we have the following estimations for the $L^p(\mathbb{R}^d)$ and $H^p _s(\mathbb{R}^d)$ norms respectively, \cite{Mey1992}, Chapter 6:
\begin{equation}\label{Lpnorm}
c_p \left\|{f}\right\|_{L^p (\mathbb{R}^d)}\leq \left\|{\left(\sum\limits_{j\,k\,e} |\langle{f,\psi_{j\,k}^e}\rangle|^2 2^{dj}\mathbf{1}_{I_{j\,k}}\right)^{\frac{1}{2}}}\right\|_{L^p (\mathbb{R}^d)}\leq C_p \left\|{f}\right\|_{L^p  (\mathbb{R}^d)}\,,
\end{equation}
and for $0\leq s\leq r$,
\begin{equation}\label{Hpnorm}
c_{p\,s} \left\|{f}\right\|_{H^p _s (\mathbb{R}^d)}\leq \left\|{\left(\sum\limits_{j\,k\,e} |\langle{f,\psi_{j\,k}^e}\rangle|^2 (1+4^{sj})2^{dj}\mathbf{1}_{I_{j\,k}}\right)^{\frac{1}{2}}}\right\|_{L^p (\mathbb{R}^d)}\leq C_{p\,s} \left\|{f}\right\|_{H^p _s (\mathbb{R}^d)}.
\end{equation}
In order to simplify the notation involving wavelet expansions we will sometimes omit the summation limits as in equations \eqref{Lpnorm} and \eqref{Hpnorm}.
\section{Main Results.}
\subsection{Convergence.}
First, we prove an inequality involving the $l^p$ norm of the wavelet coefficients of a function. As a byproduct, this inequality implies one case of the Sobolev's embeddings, see e.g. \cite{Adams}, Theorem 7.57.
\begin{sat}\label{T1}
Let $\{\psi_{j\,k}^e \}_{j\,k\,e}$ be an $r$-regular orthonormal wavelet basis,  $1 < p < 2$ and $ d\left({\dfrac{1}{p}-\dfrac{1}{2}}\right)<s<r$ then there exists a positive constant $C_{p\,s}$ such that:
\begin{equation}\label{e3}
\left\|f\right\|_{L^2(\mathbb{R}^d)} \leq \left({\sum\limits_{j\,k\,\,e}|\langle{f,\psi_{j\,k} ^e}\rangle|^{p}  }\right)^{\frac{1}{p}} \leq C_{p\,s} \left\|f\right\|_{H^p _s (\mathbb{R}^d)}\,,
\end{equation}
for all $f\in  H^p _s (\mathbb{R}^d)$. If $p=2$, the inequality \eqref{e3} holds for $s\geq 0$.
\end{sat}
\begin{proof} The case $p=2$ is immediate since
$ \left\|f\right\|_{L^2(\mathbb{R}^d)} \leq  \left\|f\right\|_{H^2 _s (\mathbb{R}^d)}$.
If $1 < p\leq 2$, the lower bound holds, since
$$\sum\limits_{j\,k\,e} |\langle{f,\psi_{j\,k} ^e}\rangle|^{p} \geq \left({\sum\limits_{j\,k\,e} |\langle{f,\psi_{j\,k} ^e}\rangle|^{2} }\right)^{\frac{p}{2}}= \left\|f\right\|_{L^2(\mathbb{R}^d)} ^p\,.$$
The upper bound is obtained splitting the sum:
$$\sum\limits_{j\,k,\,e} |\langle{f,\psi_{j\,k} ^e}\rangle|^{p} =\sum\limits_{j< 1\,k\,e} |\langle{f,\psi_{j\,k} ^e}\rangle|^{p} +\sum\limits_{j \geq 1\,k\,e} |\langle{f,\psi_{j\,k} ^e}\rangle|^{p}\,.$$
Then for each $e\in E$:
$$\sum\limits_{j\geq1\,k} |\langle{f,\psi_{j\,k} ^e}\rangle|^{p} =\int\limits_{\mathbb{R}^d}\sum\limits_{j\geq 1} 2^{jd\left({1-\frac{p}{2}}\right)} 4^{-js\frac{p}{2}} 4^{js\frac{p}{2}}\left({\sum\limits_{k\in\mathbb{Z}^d} |\langle{f,\psi_{j\,k} ^e}\rangle|^{2}\mathbf{1}_{I_{j\,k}}(x) 2^{jd} }\right)^{\frac{p}{2}} dx\,,$$
since for fixed $j$, $I_{j\,k}\bigcap I_{j\,k'}=\emptyset$ if $k\neq k'$. The inner integrand can be rewritten as
$$\sum\limits_{j\geq 1} 2^{jd\left({(1-\frac{p}{2})-\frac{sp}{d}}\right) }4^{js\frac{p}{2}} \left({\sum\limits_{k\in\mathbb{Z}^d} |\langle{f,\psi_{j\,k} ^e}\rangle|^{2}\mathbf{1}_{I_{j\,k}} 2^{jd} }\right)^{\frac{p}{2}}$$
$$\leq\left({\sum\limits_{j\geq 1} \left({2^{jd\left({(1-\frac{p}{2})-\frac{sp}{d}}\right)}}\right)^{\frac{2}{2-p}}}\right)^{\frac{2-p}{2}} \left({\sum\limits_{j\geq 1} 4^{js} \sum\limits_{k\in\mathbb{Z}^d} |\langle{f,\psi_{j\,k} ^e}\rangle|^{2}\mathbf{1}_{I_{j\,k}} 2^{jd} }\right)^{\frac{p}{2}}\,,$$
by H\"{o}lder's inequality with exponents $\dfrac{2}{p}$ and $\dfrac{2}{2-p}$ and since $s> d\left({\dfrac{1}{p}-\dfrac{1}{2}}\right)$. Hence
 \begin{equation}\label{e4}\sum\limits_{j\geq 1\,k} |\langle{f,\psi_{j\,k} ^e}\rangle|^{p}\leq C_{d\,p\,s} \int\limits_{\mathbb{R}^d} \left({\sum\limits_{j\geq 1} 4^{js} \sum\limits_{k\in\mathbb{Z}^d} |\langle{f,\psi_{j\,k} ^e}\rangle|^{2}\mathbf{1}_{I_{j\,k}} (x) 2^{jd} }\right)^{\frac{p}{2}} dx
 \end{equation}
$$\sum\limits_{j\geq1\,k} |\langle{f,\psi_{j\,k} ^e}\rangle|^{p}\leq C'_{d\,p\,s}\left\|f\right\|_{H^p _s (\mathbb{R}^d)} ^p\,.$$
For the bound on the  other term, we proceed similarly to the previous case:
$$\sum\limits_{j<1\,k} |\langle{f,\psi_{j\,k} ^e}\rangle|^{p}=\int\limits_{\mathbb{R}^d}\sum\limits_{j< 1} 2^{jd\left({1-\frac{p}{2}}\right)} \left({\sum\limits_{k\in\mathbb{Z}^d} |\langle{f,\psi_{j\,k} ^e}\rangle|^{2}\mathbf{1}_{I_{j\,k}}(x) 2^{jd} }\right)^{\frac{p}{2}} dx\,.$$
Therefore by by H\"{o}lder's inequality with exponents $\dfrac{2}{p}$ and $\dfrac{2}{2-p}$, if $$C''_{d\,p}=\left({\sum\limits_{j< 1} 2^{jd\left({1-\frac{p}{2}}\right)\frac{2}{2-p}}} \right)^{\frac{2-p}{p}}\,,$$ we get
$$\sum\limits_{j<1\,k} |\langle{f,\psi_{j\,k} ^e}\rangle|^{p}\leq C''_{d\,p} \int\limits_{\mathbb{R}^d} \left({\sum\limits_{j< 1}\sum\limits_{k\in\mathbb{Z}^d} |\langle{f,\psi_{j\,k} ^e}\rangle|^{2}\mathbf{1}_{I_{j\,k}}(x) 2^{jd} }\right)^{\frac{p}{2}} dx$$
\begin{equation}\label{e5}
\leq C''_{d\,p} \left\|f\right\|_{L^p (\mathbb{R}^d)}\leq C''_{d\,p} \left\|f\right\|_{H^p _s (\mathbb{R}^d)}
\end{equation}
Combining equations \eqref{e4} and \eqref{e5} and since $E$ is finite we get the result.
\end{proof}
Now, we can prove one of the main results of this work.
\begin{sat}\label{T2} Let $\{\psi_{j\,k}^e \}_{j\,k\,e}$ be an $r$-regular orthonormal wavelet series, with $ d\left({\dfrac{1}{p}-\dfrac{1}{2}}\right)<\gamma \leq d\left({1-\dfrac{1}{p}}\right)$, $\dfrac{3}{4}\leq p\leq 2$, $\gamma<r$ and $(\eta_{j\,k\,e})_{j\,k\,e}$ a sequence of independent identically distributed random variables such that $\eta_{j\,k\,e} \sim SpS$. Then the series defined by
$$X_{\gamma}=\sum\limits_{j\,k\,e} \eta_{j\,k\,e}  \mathcal{I}_{\gamma} \psi_{j\,k} ^e$$
converges a.s.  in $\mathcal{D}'(\mathbb{R}^d)$. If $p=2$, the result remains true for $0\leq\gamma \leq \dfrac{d}{2}$.
\end{sat}

\begin{proof} We shall prove the case $p<2$, the $p=2$ case is very similar using Parseval's identity instead of Theorem \ref{T1}. Let $Q=\left[{\dfrac{-1}{4},\dfrac{1}{4}}\right)^d$,
since $( \mathcal{I}_{\gamma} \psi_{j\,k} ^e)\mathbf{1}_Q\,\in\,L^2 (\mathbb{R}^d)$, then by lemma \ref{L1},
$$
\left\|{( \mathcal{I}_{\gamma} \psi_{j\,k} ^e)\mathbf{1}_Q}\right\|_{\mathcal{F}{L^p}_w}\leq C_{p\,d} \sum\limits_{n\in\mathbb{Z}^d} |\widehat{( \mathcal{I}_{\gamma} \psi_{j\,k} ^e)\mathbf{1}_Q}(n)|^p (1+|n|^2)^{-d}\,,
$$
thus
$$
\sum\limits_{j \,k\,e}\left\|{( \mathcal{I}_{\gamma} \psi_{j\,k} ^e)\mathbf{1}_Q}\right\|_{\mathcal{F}{L^p}_w} ^p\leq C_{p\,d} \sum\limits_{n\in\mathbb{Z}^d} \sum\limits_{j \,k\,e} |\widehat{( \mathcal{I}_{\gamma} \psi_{j\,k} ^e)\mathbf{1}_Q}(n)|^p (1+|n|^2)^{-d}\,
$$
\begin{equation}\label{e6}=C_{p\,d} \sum\limits_{n\in\mathbb{Z}^d} (1+|n|^2)^{-d}\sum\limits_{j \,k\,e} |\widehat{( \mathcal{I}_{\gamma} \psi_{j\,k} ^e)\mathbf{1}_Q}(n)|^p \,.
\end{equation}
But, if $e_n (x)= \mathbf{1}_Q (x) e^{i2\pi n x}$,  a density argument applied to equation \eqref{selfadj} gives:
$$\widehat{( \mathcal{I}_{\gamma} \psi_{j\,k} ^e)\mathbf{1}_Q}(n)=\langle{( \mathcal{I}_{\gamma} \psi_{j\,k} ^e)\mathbf{1}_Q, e_n}\rangle=\langle{ \psi_{j\,k} ^e, \mathcal{I}_{\gamma} e_n}\rangle=\overline{\langle{  \mathcal{I}_{\gamma} e_n,\psi_{j\,k} ^e}\rangle}\,.
$$
Therefore, by Theorem \ref{T1}, and taking $\gamma=s$ :
\begin{equation}\label{e7}
\sum\limits_{j \,k\,e} |\widehat{( \mathcal{I}_{\gamma} \psi_{j\,k} ^e)\mathbf{1}_Q}(n)|^p=\sum\limits_{j \,k\,e} |\langle{  \mathcal{I}_{\gamma} e_n,\psi_{j\,k} ^e}\rangle|^p \leq C_{p\,s}\left\| \mathcal{I}_{\gamma} e_n\right\|_{H^p _s (\mathbb{R}^d)}
\end{equation}
\begin{equation}\label{e8}
\leq C'_{p\, s}(\left\| \mathcal{I}_{\gamma} e_n\right\|_{L^p  (\mathbb{R}^d)}+\left\|I_{\gamma-s} e_n\right\|_{L^p  (\mathbb{R}^d)}) \leq C'_{p\,\gamma} (\left\|e_n\right\|_{L^r  (\mathbb{R}^d)}+\left\| e_n\right\|_{L^p  (\mathbb{R}^d)})\,.
\end{equation}
The last inequality holds by the Hardy-Littlewood and Sobolev Inequality with exponents $\dfrac{1}{r}-\dfrac{1}{p}=\dfrac{\gamma}{d}$.
Note that the validity of this last step is granted since $\dfrac{4}{3}\leq p\leq 2$ and $ d\left({\dfrac{1}{p}-\dfrac{1}{2}}\right)\leq\gamma \leq d\left({1-\dfrac{1}{p}}\right)$. Moreover $\left\|e_n\right\|_{L^r  (\mathbb{R}^d)}+\left\| e_n\right\|_{L^p  (\mathbb{R}^d)}$ is finite and constant in $n$. Thus from the definition of $\mathcal{F}{L^p}_w$ combined with equations \eqref{e8}, \eqref{e7} and \eqref{e6}:
\begin{equation}\label{e9}
\int\limits_{\mathbb{R}^d}\sum\limits_{j \,k\,e}\left|{\widehat{( \mathcal{I}_{\gamma} \psi_{j\,k} ^e)\mathbf{1}_Q} (\lambda)}\right|^p (1+|\lambda|^2)^{-d} d\lambda
\end{equation}
$$
=\sum\limits_{j \,k\,e}\left\|{( \mathcal{I}_{\gamma} \psi_{j\,k} ^e)\mathbf{1}_Q}\right\|_{\mathcal{F}{L^p}_w} ^p \leq C_{p\,d}\sum\limits_{n\in\mathbb{Z}^d} (1+|n|^2)^{-d}\sum\limits_{j \,k\,e} |\widehat{( \mathcal{I}_{\gamma} \psi_{j\,k} ^e)\mathbf{1}_Q}(n)|^p <\infty\,.
$$
Taking any $1<r<p$, by H\"{o}lder's inequality combined with equation \eqref{e9}:
$$
\int\limits_{\mathbb{R}^d}\left({\sum\limits_{j \,k\,e}\left|{\widehat{( \mathcal{I}_{\gamma} \psi_{j\,k} ^e)\mathbf{1}_Q} (\lambda)}\right|^{p}}\right)^{\frac{r}{p}} (1+|\lambda|^2)^{-d} d\lambda
$$
$$
\leq\left({\int\limits_{\mathbb{R}^d}\sum\limits_{j \,k\,e}\left|{\widehat{( \mathcal{I}_{\gamma} \psi_{j\,k} ^e)\mathbf{1}_Q} (\lambda)}\right|^p (1+|\lambda|^2)^{-d} d\lambda}\right)^{\frac{r}{p}} \left({\int\limits_{\mathbb{R}^d} \frac{1}{(1+|\lambda|^2)^d} d\lambda }\right)^{1-\frac{r}{p}}<\infty\,
$$
then, by Theorem \ref{Stoconv},
$$
\sum\limits_{j \,k\,e}\eta_{j\,k\,e}\widehat{( \mathcal{I}_{\gamma} \psi_{j\,k} ^e)\mathbf{1}_Q}
$$
converges a.s. in $L^r(\mathbb{R}^d, w \, d\lambda)$ and therefore $\sum\limits_{j \,k\,e} \eta_{j\,k\,e}{( \mathcal{I}_{\gamma} \psi_{j\,k} ^e)\mathbf{1}_Q} $ converges a.s. in
$\mathcal{F}{L^r}_w$ and in $\mathcal{S}'(\mathbb{R}^d)$. With slight modifications, the same argument works with any translate of $Q$. Finally, to verify that $\sum\limits_{j \,k\,e} \eta_{j\,k\,e}{ \mathcal{I}_{\gamma} \psi_{j\,k} ^e} $ converges a.s. in $\mathcal{D}'(\mathbb{R}^d)$, take
$\mathcal{Q}=\left\{{Q=\left[{\dfrac{-1}{4},\dfrac{1}{4}}\right)^d +\dfrac{k}{2},\,k\in\mathbb{Z}^d}\right\}$
, $\Omega'$ with $\mathbf{P}(\Omega')=1$ defined by
$$
\Omega' = \mathop{\bigcap}\limits_{Q\in\mathcal{Q}} \left\{{\omega\in\Omega\,:\,\left\| {\sum\limits_{j \,k\,e} \eta_{j\,k\,e} (\omega){( \mathcal{I}_{\gamma} \psi_{j\,k} ^e)\mathbf{1}_Q}} \right\|_{\mathcal{F}_r}  < \infty}\right\}
$$
and $\varphi\in \mathcal{D}(\mathbb{R}^d)$. For fixed $Q\in\mathcal{Q}$, $\omega\in\Omega'$ and $N,M \in \mathbb{N}$
we have
$$
s_{N\,M\,Q}(\omega)=\sum\limits_{|j|\leq N \,|k|\leq M}\sum\limits_{e\in E} \eta_{j\,k\,e}(\omega){( \mathcal{I}_{\gamma} \psi_{j\,k} ^e)\mathbf{1}_Q} \in L^2(\mathbb{R}^d)\,,
$$
and then
$$
\left\langle{\sum\limits_{Q} s_{N\,M\,Q} (\omega),\varphi}\right\rangle=
\sum\limits_{i=1}^l \left\langle{ s_{N\,M\,Q_i}(\omega),\varphi}\right\rangle
$$
for some $Q_i$ such that $supp(\varphi) \subset \mathop{\bigcup}\limits_{i=1}^m Q_i$ since $\varphi$ has compact support. The result follows from the convergence of  $\langle{ s_{N\,M\,Q_i}(\omega),\varphi}\rangle$ when $N,M \longrightarrow\infty$ for each $i=1\dots m$.
\end{proof}
Alternatively, considering $\gamma>\dfrac{d}{2}$ and the operators $\mathcal{K}_{\gamma}$ instead of $\mathcal{I}_{\gamma}$ we can prove:
\begin{sat}\label{T2b} Let $\{\psi_{j\,k}^e \}_{j\,k\,e}$ be an $r$-regular orthonormal wavelet series, $ \dfrac{d}{2}<\gamma \leq d\left({1-\dfrac{1}{p}}\right)+1$, $1 \leq p\leq 2$, $\gamma<r$ and $(\eta_{j\,k\,e})_{j\,k\,e}$ a sequence of independent identically distributed random variables such that $\eta_{j\,k\,e} \sim SpS$. Then, for each $x\in\mathbb{R}^d$ the series defined by
$$Y_{\gamma}(x)=\sum\limits_{j\,k\,e} \eta_{j\,k\,e}  \mathcal{K}_{\gamma} \psi_{j\,k} ^e (x)$$
converges almost surely. Moreover, $\{Y_{ \gamma}(x)\}_{x\in\mathbb{R}^d}$ has a measurable version. If $p=2$, the result remains true for $\dfrac{d}{2}\leq\gamma \leq \dfrac{d}{2}+1$.
\end{sat}

\subsubsection*{Remark.} Note that the range of validity of the result depends on the dimension $d$, since the restrictions imply that $1<\dfrac{2d}{d+2}<p\leq 2$ for $d\geq 2$.
\begin{proof}
Recall the properties of the $p$ stable random variables reviewed in Section \ref{secsta}. For each $x\in\mathbb{R}^d$, we can prove the convergence in $r$-mean
($r<p$) of the sum defining $Y_{\gamma}(x)$. By Theorem \ref{T1}, and taking any $s$ such that $d\left({\dfrac{1}{p}-\dfrac{1}{2}}\right)<s<\gamma-d\left({1-\dfrac{1}{p}}\right)$, since $ \mathcal{K}_{\gamma} \psi_{j\,k} ^e (x)=\langle{K_{\gamma}(x,\,.\,),\psi_{j\,k} ^e}\rangle$ for some constant $C$. we obtain:
$$(\mathbf{E}|Y_{\gamma}(x)|^r)^{\frac{1}{r}}= C \left({\sum\limits_{j\,k\,e}|\langle{K_{\gamma}(x,\,.\,),\psi_{j\,k} ^e}\rangle|^{p}  }\right)^{\frac{1}{p}}\leq C' \left\|K_{\gamma}(x,\,.\,)\right\|_{H^p _s (\mathbb{R}^d)}<\infty\,,$$
since, recalling from Section \ref{auxfunc} the Lemma \ref{ker}, and the equivalence of norms of $H^p _s (\mathbb{R}^d)$ given by equation \eqref{normadeK}, one obtains: $$\left\|K_{\gamma}(x,\,.\,)\right\|_{H^p _s (\mathbb{R}^d)} \leq\,C(\left\|K_{\gamma-s}(x,\,.\,)\right\|_{L^p(\mathbb{R}^d)}+ \left\|K_{\gamma}(x,\,.\,)\right\|_{L^p(\mathbb{R}^d)}) $$
$$\leq C' ( |x|^{(\gamma-s) -\left({1-\frac{1}{p}}\right)d}+ |x|^{\gamma -\left({1-\frac{1}{p}}\right)d})\,.$$
The sum defining $Y_{\gamma} (x)$ converges a.s. since convergence in the $r$-mean of independent random variables implies a.s. convergence.
Similarly to the previous bound, if $|x-x'|<1$, by Lemma \ref{ker} (ii) one gets:
$$(\mathbf{E}|Y_{\gamma}(x)-Y_{\gamma}(x')|^r)^{\frac{1}{r}}$$
$$=C\left({\sum\limits_{j\,k\,e}|\langle{K_{\gamma}(x,\,.\,)-K_{\gamma}(x'),\psi_{j\,k} ^e}\rangle|^{p}  }\right)^{\frac{1}{p}}= C\left({\sum\limits_{j\,k\,e}|\langle{K_{\gamma}(x-x',\,.\,),\psi_{j\,k} ^e}\rangle|^{p}  }\right)^{\frac{1}{p}}$$
$$\leq C' |x-x'|^{(\gamma-s) -\left({1-\frac{1}{p}}\right)d}\,,$$
From this, applying Tchebychev's inequality,  it follows the stochastic continuity of $Y_{\gamma}(x)$, and then there exists a measurable version (Theorem 1, p.157 of \cite{Gikh}) of $\{Y_{\gamma}(x)\}_{x\in\mathbb{R}^d}$.
\end{proof}
\subsection{Self similarity analysis}
Self similarity in the sense of equation \eqref{self} is broken if $p\neq 2$. However, the following results show that, in some sense, the rescaled versions of $X_{\gamma}$ are stochastically dominated. Furthermore, we may expect some
kind of fractal behavior for an integrated version of $X_{\gamma}$, as the realizations of $Y_{\gamma}$ considering a Daubechies wavelet basis suggest, see Figures \ref{fig:f0} and \ref{fig:f1}.
\begin{figure}[h]
\centering
\subfigure[$p=2$]{
\includegraphics[width=6.5cm,height=6.5cm]{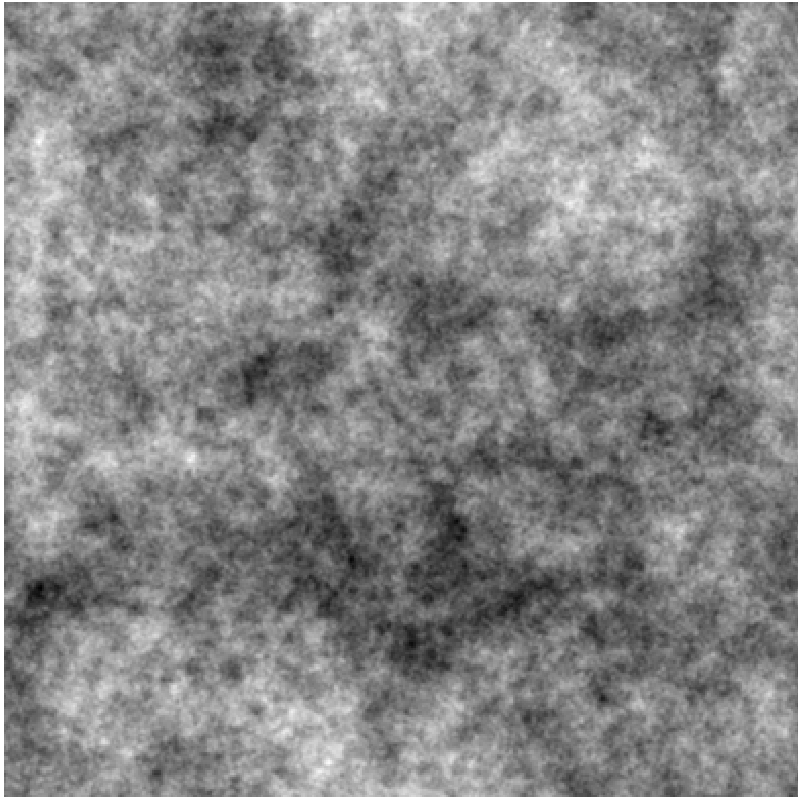} }
\subfigure[$p=1.8$]{
\includegraphics[width=6.5cm,height=6.5cm]{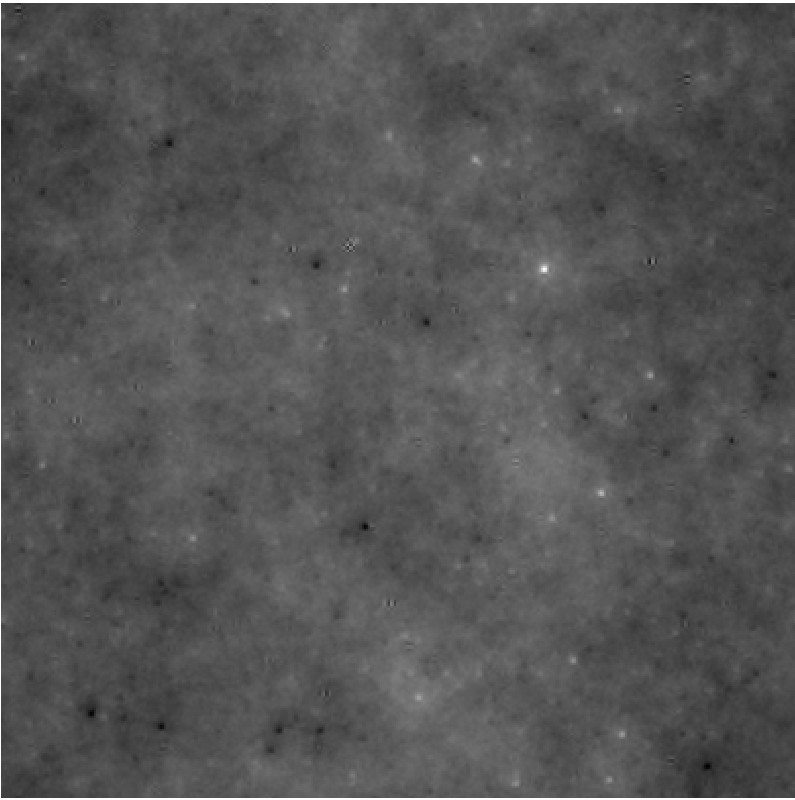} }
\caption{Observations of $Y_{\gamma}$, $d=2$ and $\gamma=1.1$.}
\label{fig:f0}
\end{figure}
\begin{figure}[h]
\centering
\subfigure[$p=2$]{
\includegraphics[width=6.5cm,height=6.5cm]{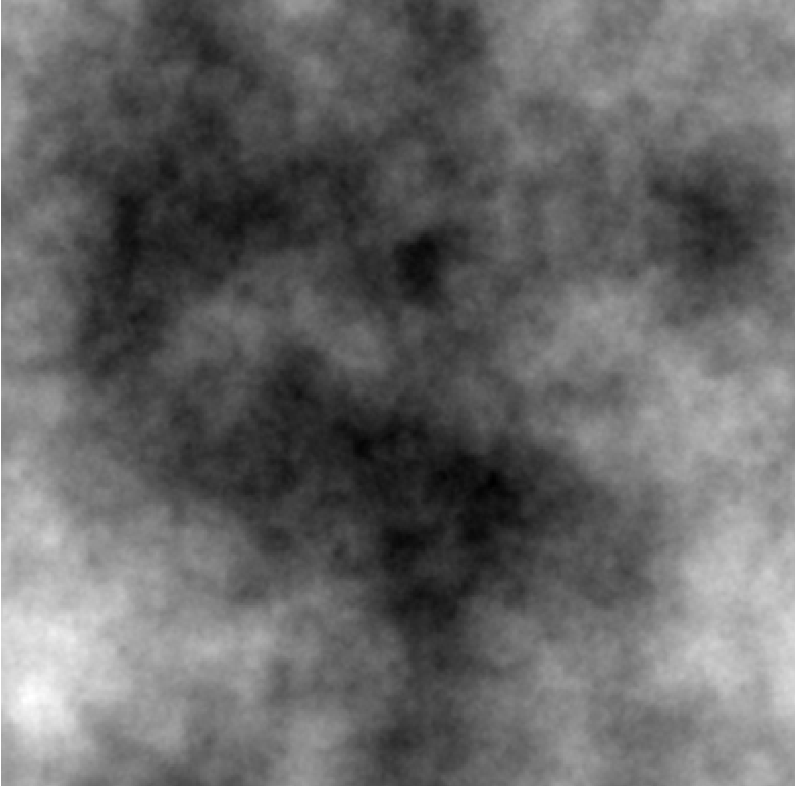} }
\subfigure[$p=1.8$]{
\includegraphics[width=6.5cm,height=6.5cm]{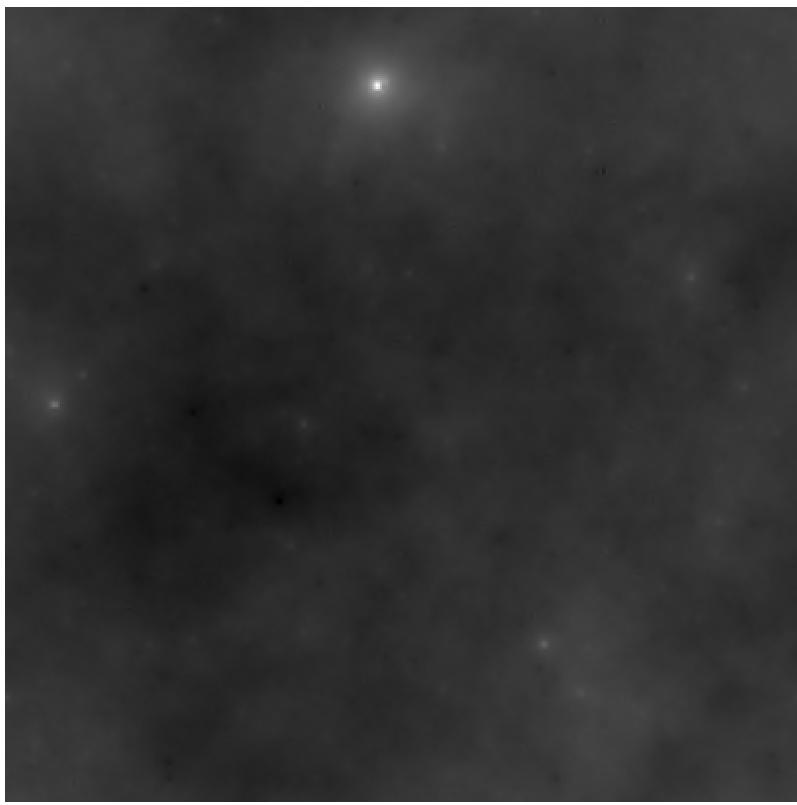} }
\caption{\label{im1}Observations of $Y_{\gamma}$, $d=2$ and $\gamma=1.6$.}
\label{fig:f1}
\end{figure}
\begin{sat}\label{T3}
Under the same hypothesis of Theorem \ref{T2}, the generalized random process $X_{\gamma}$ defined by:
\begin{equation}\label{e10}
X_{\gamma}=\sum\limits_{j \,k\,e} \eta_{j\,k\,e}{ \mathcal{I}_{\gamma} \psi_{j\,k} ^e}
\end{equation}
is self similar if $p=2$, in the sense that for every $\varphi\in\mathcal{D}(\mathbb{R}^d)$, $a^{\frac{d}{2}+\gamma}\langle{X_{\gamma},\varphi(a\,.\,)}\rangle$ has the same distribution function as $\langle{X_{\gamma},\varphi}\rangle$, and otherwise, for every $s>\gamma$, there exists a positive constant $C_{p\,s}$ such that the   following bounds hold:
\begin{equation}\label{e11}
 F_{\eta_p}(C_{p\,s} (a^{d\left({\frac{1}{2}-\frac{1}{p}}\right)}\left\| \mathcal{I}_{\gamma}\varphi\right\|_{L^p (\mathbb{R}^d)}+ a^{d\left({\frac{1}{2}-\frac{1}{p}}\right)+s}\left\|\mathcal{I}_{\gamma-s}\varphi\right\|_{L^p (\mathbb{R}^d)})^{-1} x)
\end{equation}
$$\leq \mathbf{P}(a^{\frac{d}{2}+\gamma} \langle{X_{\gamma},\varphi(a\,.\,)}\rangle\leq x) \leq F_{\eta_p}(\left\| \mathcal{I}_{\gamma}\varphi\right\|_{L^2 (\mathbb{R}^d)} ^{-1} x)\,,$$
for every $a>0$, $\varphi\in\mathcal{D}(\mathbb{R}^d)$ and $x>0$.

\end{sat}
\subsubsection*{Remark} Note that in the case $p=2$ it is easy to verify that the limit process is a Gaussian fractional noise with characteristic functional
$\phi_{X_{\gamma}}(\varphi)= e^{-\left\|\mathcal{I}_{\gamma}\varphi \right\|^2 _{L^2(\mathbb{R}^d)}}$, and that this stationary generalized random process has a \emph{spectral measure}, \cite{Gel}, Chapter 3, given by $d\mu_{X_{\gamma}} (\lambda)= (2\pi)^{-2\gamma}\dfrac{d\lambda}{|\lambda|^{2\gamma}}$. However, if $p\neq 2$, the analogous result for the stable case does not hold, since   $\phi_{X_{\gamma}}(\varphi)\neq e^{-\left\|\mathcal{I}_{\gamma}\varphi \right\|^p _{L^p(\mathbb{R}^d)}}$, which corresponds to the case of fractional stable noise.
\begin{proof}
Let $p<2$ and $\varphi\in\mathcal{D}(\mathbb{R}^d)$. To prove equation \eqref{e11} it is sufficient to analyze $\Phi_{X(\varphi(a\,.\,))}$, the characteristic function of the real random variable $a^{\frac{d}{2}+\gamma}\langle{X_{\gamma},\varphi(a\,.\,)}\rangle$. From the scaling property of $\mathcal{I}_{\gamma}$:
$$a^{\frac{d}{2}+\gamma} \langle{X_{\gamma},\varphi(a\,.\,)}\rangle=a^{\frac{d}{2}}\sum\limits_{j \,k\,e} \eta_{j\,k\,e}\langle{ \psi_{j\,k} ^e, (\mathcal{I}_{\gamma}\varphi)(a\,.\,)}\rangle\,.$$
Assume $\sigma=1$ with no loss of generality. Since the $\eta_{j\,k\,e}$'s are independent and identically distributed with characteristic function $\Phi_{\eta_{j\,k\,e}}(\xi)= e^{-|\xi|^p}$, then the sum defining $a^{\frac{d}{2}+\gamma}\langle{X_{\gamma},\varphi(a\,.\,)}\rangle$ has characteristic function given by:
\begin{equation}\label{e12}Ln\left({\Phi_{a^{\frac{d}{2}+\gamma}X(\varphi(a\,.\,))} (\xi)}\right)={-a^{\frac{dp}{2}}\left({\sum\limits_{j\,k\,e}|\langle{(\mathcal{I}_{\gamma}\varphi)(a\,.\,),\psi_{j\,k} ^e}\rangle|^{p}  }\right)|\xi|^p}\,
\end{equation}
which corresponds to the distribution $$F_{\eta_p}\left({a^{\frac{-d}{2}}\left({\sum\limits_{j\,k\,e}|\langle{(\mathcal{I}_{\gamma}\varphi)(a\,.\,),\psi_{j\,k} ^e}\rangle|^{p}  }\right)^{-1/p} x}\right)\,.$$
\noindent
Then, the upper bound follows combining Theorem \ref{T1} and the fact that $F_{\eta_p}$ is monotone. The lower bound is obtained similarly estimating the norm
$$\left\|(\mathcal{I}_{\gamma}\varphi)(a\,.\,)\right\|_{H^p _s (\mathbb{R}^d)}\,.$$ Finally, the case $p=2$ is obtained in an analogous way with equality due to Parseval's identity for the orthonormal basis $\{\psi_{j\,k} ^e\}_{j\,k\,e}$ of $L^2(\mathbb{R}^d)$.
\end{proof}
 The previous result is a consequence of the bound derived from Theorem \ref{T1}:
 \begin{equation}\label{e13}
 Ln\left({\Phi_{X(\varphi)} (\xi)}\right)=-\left({\sum\limits_{j\,k\,e}|\langle{(\mathcal{I}_{\gamma}\varphi),\psi_{j\,k} ^e}\rangle|^{p}  }\right)|\xi|^p\leq -\left\|\mathcal{I}_{\gamma}\varphi\right\|^p _{L^2(\mathbb{R}^d)}|\xi|^p\,.
 \end{equation}
 For $x\in\mathbb{R}^d$, and taking a sequence $\varphi_{n\,x} \in\mathcal{D}(\mathbb{R}^d)$ such that $\mathop{\varphi_ {n\,x} \longrightarrow K_{\beta}(x,\,.\,)}$ in $L^p (\mathbb{R}^d)$ as $n\longrightarrow\infty$, provided that $\gamma +\beta$ are as in Theorem \ref{T2b},  we can interpret $Y_{\gamma+\beta}$ as an integrated observation of $X_{\gamma}$: $Y_{\gamma+\beta}(x)=\langle{X_{\gamma},K_{\beta}(x,\,.\,)} \rangle=\int\limits_{\mathbb{R}^d} K_{\beta} (x,y) X_{\gamma}(y) dy$, where these equalities are only formal.
 In fact $Y_{\gamma+\beta}(x)$ is a well defined ordinary random variable for each $x\in\mathbb{R}^d$. Recalling equation \eqref{combfrac} and Section \ref{secsta}, its characteristic function is given by $$Ln(\Phi_{Y_{\gamma+\beta}(x)}(\xi))=-\left({\sum\limits_{j\,k\,e}|\langle{({K}_{\gamma+\beta}(x,\,.\,),\psi_{j\,k} ^e }\rangle|^{p}  }\right)|\xi|^p$$ which is the pointwise limit of the sequence of characteristic functions $$\{\Phi_{\langle{X_\gamma,\varphi_{n\,x}}\rangle}(\xi)\}_{n\in\mathbb{N}}\,.$$
This is a consequence of the following bound, which again can be derived from Theorem \ref{T1} with $s=\gamma$:
$$\left|{(-Ln(\Phi_{Y_{\gamma+\beta}(x)}(\xi)))^{1/p} -(-Ln(\Phi_{\langle{X_\gamma,\varphi_{n\,x})}\rangle}(\xi)))^{1/p}}\right|$$
$$\leq |\xi|\left({ \sum\limits_{j\,k\,e}|\langle{(\mathcal{I}_{\gamma}({K}_{\beta}(x,\,.\,)-\varphi_{n\,x})),\psi_{j\,k} ^e }\rangle|^{p}}\right)^{\frac{1}{p}}$$
$$\leq C_{p\,s}|\xi|(\left\| \mathcal{I}_{\gamma} ({K}_{\beta} (x,\,.\,)-\varphi_{n\,x})\right\|_{L^p  (\mathbb{R}^d)}+\left\|{{K}_{\beta} (x,\,.\,)-\varphi_{n\,x}}\right\|_{L^p  (\mathbb{R}^d)}) \,.$$

 The Lebesgue measure in $\mathbb{R}^{d+1}$ of a measurable version of $\{Y_{\gamma}(x)\}_{x\in\mathbb{R}^d}$ is zero. Let us bound, from below,   the Hausdorff dimension of the graph $\mathcal{G}\subset\mathbb{R}^{d+1}$ of $Y_{\gamma}(x)$.
As a consequence, we shall see that for suitable parameters, the Hausdorff dimension has non integer values.
 \begin{sat}
 Under the same hypothesis of Theorem \ref{T2b},  then $\dfrac{3d}{2}-\gamma+1 \leq dim_H(\mathcal{G})$ a.s., where  $\mathcal{G}\subset\mathbb{R}^{d+1}$ is the graph of $Y_{\gamma}(x)$.
 \end{sat}
 \begin{proof}


The lower bound is a consequence of Lemma \ref{Frost}. We shall prove that $$\mathbf{E}\int\limits_{B}\int\limits_{B} (|x-x'|^2+|Y(x)-Y(x')|^2)^{-\rho/2} \, dx \, dx' <\infty$$ if  $\rho<\dfrac{3d}{2}-\gamma+1$. Let us write $\Delta(x,x')=Y(x)-Y(x')$, then recalling equation \eqref{e13}, by Lemma \ref{ker}, (i) and (ii), one gets:
\begin{equation}\label{e14}
- Ln\left({\Phi_{\Delta} (\xi)}\right)=\left({\sum\limits_{j\,k\,e}|\langle{K_{\gamma}(x-x',\,.\,),\psi_{j\,k}^e }\rangle|^{p}  }\right)|\xi|^p
 \end{equation}
$$\geq \left\|K_{\gamma}(x-x',\,.\,) \right\|^p _{L^2(\mathbb{R}^d)}|\xi|^p= C \left({|x-x'|^{\gamma-\frac{d}{2}}}\right)^{p}|\xi|^p \,.$$
Hence, from equation \eqref{e14} :
$$\mathbf{E}((|x-x'|^2+|Y(x)-Y(x')|^2)^{-\rho/2})\leq \int\limits_{\mathbb{R}}\int\limits_{\mathbb{R}} \frac{1}{(|x-x'|^2+|u|^2)^{\rho/2}}{|\Phi_{\Delta}(\xi)|} d\xi du  $$
$$\leq \int\limits_{\mathbb{R}} \frac{1}{(|x-x'|^2+|u|^2)^{\rho/2}}du \int\limits_{\mathbb{R}}e^{-|\xi|^p |x-x'|^{p\gamma-\frac{pd}{2}}}  d\xi \leq \frac{C}{|x-x'|^{\rho-1+\gamma -\frac{d}{2}}}\,, $$
and therefore, if for example without loss of generality $B=\{|x|\leq 1\}$,
$$\mathbf{E}\int\limits_{B}\int\limits_{B} (|x-x'|^2+|Y(x)-Y(x')|^2)^{-\rho/2} \, dx \, dx' $$
$$\leq C \int\limits_{B}\int\limits_{B} \frac{1}{|x-x'|^{\rho-1+\gamma -\frac{d}{2}}} \, dx \, dx' <\infty$$
provided that $\rho<\dfrac{3d}{2}-\gamma+1$, which concludes the proof.
 \end{proof}
 \section*{Acknowledgment.}
The authors thank the collaboration of  Alexandre Chevallier, visiting student from the
\'Ecole Internationale des Sciences
du Tra\^{\i}tement de l'Information, \'Ecole d'Ing\'enieurs Math\'ematiques,
for the computer simulations
corresponding to Figures \ref{fig:f0} and \ref{fig:f1}.

\end{document}